%% file: coprime_sts.tex
\documentclass[sts]{imsart}

\RequirePackage[OT1]{fontenc}
\usepackage{amsthm,amsmath,natbib,amssymb,bbm,mathtools}
\RequirePackage[colorlinks,citecolor=blue,urlcolor=blue]{hyperref}
\usepackage{cleveref}

\startlocaldefs
\numberwithin{equation}{section}
\theoremstyle{plain}
\newtheorem{thm}{Theorem}[section]
\newtheorem{lem}[thm]{Lemma}
\endlocaldefs

\begin{document}

\begin{frontmatter}
\title{On the probability that two random integers are coprime}
\runtitle{Probability of coprime integers}

\begin{aug}
\author{\fnms{Jing} \snm{Lei}\thanksref{t1}\ead[label=e1]{jinglei@andrew.cmu.edu}}
\and
\author{\fnms{Joseph B.} \snm{Kadane}\ead[label=e2]{kadane@stat.cmu.edu}}

\thankstext{t1}{Research is partially supported by NSF grant DMS-1553884.}
\runauthor{Lei and Kadane}

\affiliation{Carnegie Mellon University}

\address{132 Baker Hall, Carnegie Mellon University, Pittsburgh, PA, 15213. \printead{e1,e2}.}

\end{aug}

\begin{abstract}
\input{abstract}
\end{abstract}


\end{frontmatter}

\section{Introduction}\label{sec:background}
\input{introduction}

\section{Finitely additive probability}\label{sec:finitely}
\input{finitely}

\section{Limiting relative frequency}\label{sec:rel_freq}
\input{frequency}

\section{Residue classes}\label{sec:residue-classes}
\input{residual}

\section{Shift invariance}\label{sec:shift-invariance}
\input{shift}

\section{Countably additive probability}\label{sec:countably}
\input{countably}

\section{Conclusion}
\input{conclusion}

\section*{Acknowledgements}
The authors would like to thank Boris Bukh for pointing out \Cref{thm:shift} and the idea of its proof.

\bibliographystyle{imsart-nameyear}
\bibliography{coprime}

\end{document}

%% file: abstract.tex
We show that there is a non-empty class of finitely additive probabilities on $\mathbb N^2$ such that for each member of the class, each set with limiting relative frequency $p$ has probability $p$. Hence, in that context the probability that two random integers are coprime is $6/\pi^2$. We also show that two other interpretations of ``random integer," namely residue classes and shift invariance, support any number in $[0, 6/\pi^2]$ for that probability. Finally, we specify a countably additive probability space that also supports $6/\pi^2$.

%% file: introduction.tex
For two integers $a,b$, let ${\rm gcd}(a,b)$ be the largest positive integer that evenly divides both $a$ and $b$.
It is a well-established result in number theory that
\begin{equation}\label{eq:one}
\lim_{ n \rightarrow \infty } \frac{\#\left\{(a,b)\in [n]^2,~{\rm gcd}(a,b)=1\right\}}{n^2} = \frac{6}{\pi^2}\,.
\end{equation}
\citep[Theorem 331]{hardy-wright2008}, where $[n]=\{1,2,...,n\}$. They then write 
``it is natural" to interpret \eqref{eq:one} as a probability, and conclude (Theorem 332) that 
\begin{equation}\label{eq:claim}
\text{the \textit{probability} that two randomly chosen integers are co-prime is $6/\pi^2$.}
\end{equation}

The purpose of this paper is to reconcile \eqref{eq:claim} with modern probability theory. To do so, we first introduce the two main frameworks used in probability, and then apply them to \eqref{eq:claim}.

\subsection{A brief introduction to probability axioms} There are two different axiom systems for probability. The first, and most familiar, is countably additive. It starts with a triple of objects: $(\Omega, \mathcal B, P)$, where $\Omega$ is the basic set of objects, $\mathcal B$ is a $\sigma$-field of subsets of $\Omega$, and $P\{\cdot\}$ is a probability over elements of $\mathcal B$ satisfying
\begin{enumerate}
\item[(a)] $P \{B\} \geq 0$ for all $B \in \mathcal B$
\item[(b)] $P \{\Omega\} = 1$
\item[(c)] If $B_1, B_2,\ldots$ is a countable sequence of disjoint elements of $\mathcal B$, then
$$
P\left\{\bigcup^\infty_{i=1} B_i\right\} = \sum^\infty_{i=1} P\{B_i\}.
$$
\end{enumerate}

This is the axiom system advocated by \citet{kolmogorov1933}\footnote{For an English translation, see \citet{kolmogorov2018}.}. It has the advantage that conditional probabilities obey the tower property:
\begin{equation}\label{eq:j1}
E(X) = E\left[E(X | Y)\right],
\end{equation}
but the disadvantage that it is defined only on $\mathcal B$, and cannot, in general, be extended to the power set of $\Omega$ (\citet[][Section 3]{billingsley1995probability} shows 
the existence of non-measurable sets).

The second axiom system is finitely additive, replacing (c) above with the condition:
\begin{enumerate}
\item[(c$^\prime$)] If $B_1,B_2,\ldots, B_n$ is a finite collection of disjoint subsets of $\Omega$, then
$$
P\left\{\bigcup^n_{i=1} B_i\right\} = \sum^n_{i=1} P\{B_i\}.
$$
\end{enumerate}
This has the advantage that a finitely additive distribution can be extended to the power set, but the disadvantage that the tower property does not, in general, hold for them. This sense of probability was particularly advocated by \citet{deFinetti1937prevision}\footnote{For an English translation, see \citet{kyburg1980studies}.}.

Every countably additive probability is finitely additive, but not conversely.

\subsection{Application to coprime integers}
Let $G = \{(a,b) \in \mathbb N^2: {\rm gcd}(a,b) = 1\}$ be the set of pairs of integers that are coprime. We explore the probability of $G$ under various different settings, beginning with limiting relative frequency. Limiting relative frequency cannot be countably additive, since the limiting relative frequency of each pair of integers is zero, but the countable union of them, $\mathbb N^2$, has limiting relative frequency one. Hence, we consider first finitely additive probabilities on $\mathbb N^2$, including limiting relative frequency, in Section~\ref{sec:finitely}. Sections~\ref{sec:rel_freq}, \ref{sec:residue-classes} and \ref{sec:shift-invariance} apply finitely additive probability on $\mathbb N^2$ respectively to limiting relative frequency, residue classes and shift invariance.  Section~\ref{sec:countably} proposes a countably additive probability specification that also supports $P\{G\} = 6/\pi^2$. 

%% file: finitely.tex
\subsection{General background}
Allowing 
probabilities that are finitely but not countably additive permits 
extension to the power set at the cost that many other useful results 
that are true for countable additive probabilities do not hold.
The following theorem, from \citet{kadane-ohagan1995} (relying on results of \citet{rao-rao1983}) gives a necessary and sufficient condition for such an extension of a finitely additive probability.

\begin{thm}\label{thm:three}
Let $\mathcal{C}$ be any collection of subsets of a set $\Omega$ such that $\Omega \in \mathcal C$. Let $\mu$ be a nonnegative real function defined on $\mathcal{C}$ such that $\mu(\Omega) = 1$. Then $\mu$ can be extended to a finitely additive probability on all subsets of $\Omega$ if and only if, for all collections of sets $A_1, \ldots, A_a$ and $B_1,\ldots, B_b$ in $\mathcal{C}$,
\begin{equation}\label{eq:two}
\sum^a_{i=1} I_{A_i} \leq \sum^b_{j=1} I_{B_j}
\end{equation}
implies that
\begin{equation}\label{eq:three}
\sum^a_{i=1} \mu(A_i) \leq \sum^b_{j=1} \mu(B_j),
\end{equation}
where $I_A$ is the indicator function of $A$.
\end{thm}
A second result, also in \citet{kadane-ohagan1995}, gives upper and lower bounds on the probability of a set $D$ (not in general in $\mathcal{C}$):

\begin{thm}\label{thm:four}
Let $\mathcal{C}$ be any collection of subsets of a set $\Omega$ such that $\Omega \in \mathcal{C}$. Let $\mu$ be a nonnegative real function defined on $\mathcal{C}$ such that $\mu(\Omega) = 1$, and let $\mu$ be extendable to a finitely additive probability on all subsets of $\Omega$. Let $\mathcal{M}$ be the set of such extensions. Consider a further set $D \subset \Omega$. Then
$$
\{\mu(D): \mu \in \mathcal{M}\} = [\ell(D,\mathcal M), u(D,\mathcal M)],
$$
where $\ell(D,\mathcal M)$ ($u(D,\mathcal M)$) is the supremum (infimum) of
\begin{equation}\label{eq:four}
h^{-1} \left\{\sum^a_{i=1} \mu(A_i) - \sum^b_{j=1} \mu(B_j)\right\}
\end{equation}
over all $A_1, A_2, \ldots, A_a, B_1, B_2, \ldots, B_b \in \mathcal{C}$ and all $a, b, h = 1,2,3,\ldots,$ such that
\begin{equation}\label{eq:five}
\sum^a_{i=1} I_{A_i} - \sum^b_{i=1} I_{B_j} \leq(\geq)hI_D.
\end{equation}
\end{thm}

\subsection{Finitely additive uniform probabilities on $\mathbb N$}
While there is only one sense of uniformity on a finite set (each element has the same probability), the same is not true on $\mathbb{N}$. Several such senses have been studied in the literature.

\begin{enumerate}
  \item \emph{Limiting relative frequency.} Define 
  $$\mathcal C_F=\{C\subseteq \mathbb N:~\lim_{n\rightarrow\infty}\#(C\cap [n])/n\text{ exists}\}$$ 
  be the collection of subsets of $\mathbb N$ with a limiting relative frequency.  Then it is natural to require $\mu(C)=\lim_{n\rightarrow\infty}\#(C\cap [n])/n$ for $C\in\mathcal C_F$.  \cite{kadane-ohagan1995} proved that such a $\mu$ is extendable. We denote the collection of all such finitely additive measures by $\mathcal M_F$.
\item \emph{Shift invariance.} Another way of defining uniform measure on $\mathbb N$ is to require $\mu$ to be shift invariant. Formally, let $s:\mathbb N\mapsto\mathbb N$ be $s(x)=x+1$. Shift invariance requires $\mu(A)=\mu(s^{-1}(A))$ for all $A\subseteq \mathbb N$. Denote the set of finitely additive shift invariant probabilities by $\mathcal M_S$.
\item \emph{Residue class.} Let $\mathcal C_R$ be the residue class, consisting of sets of the form 
\begin{equation}\label{eq:R}
C=R_{j,k}=\{x:x\equiv j\mod k\}  
\end{equation}
 for some $j\in [k]-1$ and $k\in \mathbb N$.  Uniformity naturally requires that $\mu(R_{j,k})=k^{-1}$ for all $k\in\mathbb N$ and $j\in [k]-1$.  \cite{kadane-ohagan1995} proved that such a $\mu$ is extendable.  We denote the collection of all such extended finitely additive measures by $\mathcal M_R$.
\end{enumerate}
The results in \cite{kadane-ohagan1995} and \citet{schirokauer-kadane2007} jointly imply that that
$$
\mathcal M_F \subset \mathcal M_S \subset\mathcal M_R 
$$
and that each of these inclusions is strict.

\subsection{Finitely additive uniform probabilities on $\mathbb N^2$} Now we extend the three types of finitely additive uniform probabilities to $\mathbb N^2$, and present our main result for finitely additive uniform distributions. The proof of the main result and some intermediate claims, such as extendability, are deferred to later sections.
\begin{enumerate}
  \item \emph{Limiting relative frequency on $\mathbb N^2$.} Define
  \begin{equation}\label{eq:rel_freq_n2}
    \mathcal C_F^2 = \left\{C\subseteq \mathbb N^2: \lim_{n_1\wedge n_2\rightarrow\infty}\frac{\#(C\cap ([n_1]\times [n_2]))}{n_1n_2}\text{ exists}\right\}\,,
  \end{equation}
  and $\mu(C)$ be the limit in \eqref{eq:rel_freq_n2} for $C\in\mathcal C_F^2$.  \Cref{thm:rel_freq_ext} below ensures that $(\mathcal C_F^2,\mu)$ can be extended to $2^{\mathbb N^2}$.  Denote the collection of all such extensions by $\mathcal M_F^2$.
  \item \emph{Shift invariance on $\mathbb N^2$.}  For $j=1,2$, define $s_j:\mathbb N^2\mapsto\mathbb N^2$ as the shift function that increases the $j$th coordinate by one.  Denote $\mathcal M_S^2$ the set of finitely additive shift invariant probabilities on $\mathbb N^2$ (i.e., those satisfy $\mu(A)=\mu(s_j^{-1}(A))$ for all $A\subseteq \mathbb N^2$ and $j=1,2$).
  \item \emph{Residue class on $\mathbb N^2$.} Let $\mathcal C_R^2=\mathcal C_R\times \mathcal C_R$ be the residue class on $\mathbb N^2$ and define $\mathcal M_R^2$ be the set of finitely additive probabilities on $\mathbb N^2$ extended from $(\mathcal C_R^2,\mu)$ with $\mu(R_{j_1,k_1}\times R_{j_2,k_2})=(k_1k_2)^{-1}$ for all $R_{j_1,k_1},R_{j_2,k_2}\in\mathcal C_R$.
\end{enumerate}
The following lemma extends its counterpart in $\mathbb N$, with an almost identical proof.
\begin{lem}\label{lem:hiararchy_n2}
    $\mathcal M_F^2\subseteq \mathcal M_S^2 \subseteq \mathcal M_R^2\,.$
\end{lem}
It is possible to also establish strict inclusions by considering direct products of the examples given in \cite{kadane-ohagan1995,schirokauer-kadane2007}.  Now we state our main result for finitely additive probabilities.
\begin{thm}\label{thm:main_additive}
  Let $G=\{(x,y)\in\mathbb N^2:{\rm gcd}(x,y)=1\}$ be the set of pairs of positive integers that are co-prime. Then
  \begin{align}
  \ell(G,\mathcal M^2_F)= u(G,\mathcal M^2_F)=u(G,\mathcal M^2_S)=u(G,\mathcal M^2_R)=6/\pi^2\,\label{eq:main_1}
  \end{align}
  and
  \begin{align}
    \ell(G,\mathcal M^2_S)=\ell(G,\mathcal M^2_R)=0\,,\label{eq:main_2}
  \end{align}
  where the numbers $u(G,\mathcal M)$, $\ell(G,\mathcal M)$ are defined in \Cref{thm:four}.
\end{thm}
\begin{proof}[Proof of \Cref{thm:main_additive}]
The proof of \Cref{thm:main_additive} essentially contains the organization of results proved in the next three sections.

First, \Cref{lem:hiararchy_n2} implies that
\begin{align}\label{eq:chain}
  \ell(G,\mathcal M_R^2)\le \ell(G,\mathcal M_S^2)\le \ell(G,\mathcal M_F^2)\le
  u(G,\mathcal M_F^2)\le u(G,\mathcal M_S^2)\le u(G,\mathcal M_R^2)\,.
\end{align}

To prove \eqref{eq:main_1}, \Cref{thm:rel_freq_G} implies that $\ell(G,\mathcal M_F^2)=u(G,\mathcal M_F^2)=6/\pi^2$, while 
  \Cref{thm:res_bound} proves that $u(G,\mathcal M_R^2)=6/\pi^2$.  Therefore, \eqref{eq:main_1} follows from \eqref{eq:chain}.
  
Next, \eqref{eq:main_2} is a direct consequence of \eqref{eq:chain} and \Cref{thm:shift}, which proves $\ell(G,\mathcal M_S^2)=0$.
\end{proof}

\Cref{thm:main_additive} implies that if we interpret uniformity by limiting relative frequency, then $G$ has measure $6/\pi^2$ in all finitely additive uniform probabilities on $\mathbb N^2$.  However, if we interpret uniformity by either shift invariance or proportion of residue classes, then the measure of $G$ can be anywhere between $0$ and $6/\pi^2$. Both the lower and upper bounds in these cases are new.

%% file: frequency.tex
In this section we prove the subset of claims in \Cref{thm:main_additive} involving $\mathcal M_F^2$, as well as extendability of $(\mathcal C_F^2,\mu)$ where $\mu$ maps $C\in\mathcal C_F^2$ to the limiting relative frequency of $C$ as defined in \eqref{eq:rel_freq_n2}.

We first establish extendability.
\begin{thm}\label{thm:rel_freq_ext}
$(\mathcal{C}_F^2,\mu)$ can be extended to $2^{\mathbb{N}^2}$.
\end{thm}
\begin{proof}[Proof of \Cref{thm:rel_freq_ext}]
Let $A_1,\ldots, A_a$ and $B_1,\ldots, B_b$ be elements of $\mathcal C_F^2$ such that 
$$
\sum^a_{i=1} I_{A_i} \leq \sum^b_{j=1} I_{B_j}.
$$
Then for all $k_1,k_2\in \mathbb N$
$$
\sum^a_{i=1} \# (A_i \cap ([k_1] \times [k_2])) \leq \sum^b_{j=1} \# (B_j \cap ([k_1] \times [k_2]))\,.
$$
So
$$
 \sum^a_{i=1} \lim_{k_1 < k_2, k_1  \rightarrow \infty}  \frac{\#(A_i \cap ([k_1] \times [k_2]))}{k_1 k_2} \leq \sum^b_{j=1} \lim_{k_1 < k_2, k_1 \rightarrow \infty} \frac{\#(B_j \cap ([k_1] \times [k_2]))}{k_1 k_2}\,,
$$
i.e., $\sum^a_{i=1} \mu(A_i) \leq \sum^b_{j=1} \mu(B_j)$.
\end{proof}

The next result finishes the proof of the subset of claims in \Cref{thm:main_additive} involving $\mathcal M_F^2$.
\begin{thm}\label{thm:rel_freq_G}
  $$\lim_{n_1\wedge n_2\rightarrow\infty}\frac{\#(G\cap([n_1]\times [n_2]))}{n_1n_2}=6/\pi^2\,.$$  As a consequence $G\in\mathcal C_F^2$ and $\mu(G)=6/\pi^2$ for all $\mu\in\mathcal M_F^2$.
\end{thm}
\Cref{thm:rel_freq_G} is a slight generalization of a Theorem in \cite{hardy-wright2008}, which focuses on the case of $n_1=n_2$.  The proof is similar.
\begin{proof}[Proof of \Cref{thm:rel_freq_G}]
  Without loss of generality, assume $n_1\le n_2$.  Let $q_{n_1,n_2}$ be the number of pairs of integers $(a,b)\in [n_1]\times [n_2]$ such that ${\rm gcd}(a,b)=1$.
  Then 
  \begin{align*}
  q_{n_1,n_2}=&n_1n_2 -\sum_p\lfloor n_1/p \rfloor \lfloor n_2/p\rfloor +\sum_{p_1\le p_2}\lfloor n_1/(p_1p_2)\rfloor\lfloor n_2/(p_1p_2)\rfloor - ...\\
  =&\sum_{k=1}^{n_1} \nu(k)\lfloor n_1/k\rfloor\lfloor n_2/k\rfloor
  \end{align*}
  where
  $\nu(\cdot)$ is the mobius function such that $\nu(k)=(-1)^s$ when $k$ is the product of $s$ distinct primes, and $\nu(k)=0$ otherwise ($\nu(1)=1$).

  Because
  \begin{align*}
  0 \le & n_1n_2/k^2 -\lfloor n_1/k\rfloor\lfloor n_2/k\rfloor\\
  = &(n_2/k-\lfloor n_2/k\rfloor) (n_1/k)+(n_1/k-\lfloor n_1/k\rfloor)\lfloor n_2/k\rfloor
  \le (n_1+n_2)/k\,,
  \end{align*}
  we have
  \begin{align*}
    \left|\sum_{k=1}^{n_1} \nu(k) (n_1n_2/k^2)-q_{n_1, n_2}\right|=&\left|\sum_{k=1}^{n_1} \nu(k)\left(n_1n_2/k^2-\lfloor n_1/k\rfloor\lfloor n_2/k\rfloor\right)\right|\\
    \le & (n_1+n_2)\sum_{k=1}^{n_1} (1/k)=o(n_1 n_2)\,.
  \end{align*}
  So
  \begin{align*}
  \frac{q_{n_1,n_2}}{n_1n_2}&=\sum_{k=1}^{n_1}\nu(k)k^{-2}+o(1)\rightarrow 6/\pi^2\,.\qedhere
  \end{align*}
\end{proof}

The results of this section give a framework that justifies Hardy and Wright's claim that (1) implies (2). In this connection, the proof of (1) offered by \citet{abrams-paris1992} is not correct, because it relies on countable additivity of limiting relative frequency.

%% file: residual.tex
In this section we first address the extendability of $\mathcal C_R^2$, and then prove that $u(G,\mathcal M_R)=6/\pi^2$. The lower bound $\ell(G,\mathcal M_R^2)=0$ will be proved as a consequence of $\ell(G,\mathcal M_S^2)=0$, which is established in the next section.

\begin{thm}\label{thm:res_ext}
Let $\mu$ be a function defined on $\mathcal C_R^2$ satisfying $\mu(R_{j_1,k_1} \times R_{j_2,k_2}) = 1/k_1k_2$ for all $j_1,j_2,k_1,k_2\in\mathbb N$. Then $\mu$ can be extended to $2^{\mathbb{N}^2}$.
\end{thm}
\begin{proof}[Proof of \Cref{thm:res_ext}]
We first establish a 1-1 map between $R_{j_1,k_1} \times R_{j_2,k_2}$ and $R_{j_2k_1 + j_1, k_1k_2}$, which is realized by writing an arbitrary $k\in [k_1k_2]-1$ uniquely as $k=j_2k_1+j_1$ for $j_1\in [k_1]-1$ and $j_2\in [k_2]-1$.
  
Now each element of $\mathcal{C}_R^2$ can be mapped 1-1 to an element of $\mathcal C_R$. By the result of \cite{kadane-ohagan1995}, the set of residue classes can be extended. Therefore so can $\mathcal{C}_R^2$.
\end{proof}

The rest of this section focuses on proving $u(G,\mathcal M_R)=6/\pi^2$.  We begin by introducing a general way of identifying $u(D,\mathcal M_R^2)$
 for arbitrary $D\subseteq\mathbb N^2$.

\begin{thm}\label{thm:res_bound}
  For all $D\subseteq\mathbb N^2$,
    $$
u(D,\mathcal M_R^2) = \inf_{t_1,t_2} \frac{r_{t_1,t_2}(D)}{t_1t_2}\,,
  $$
  where
  \begin{align}
  r_{t_1,t_2}(D)=&\#\{(k_1,k_2)\in([t_1]-1)\times([t_2]-1):~D\cap (R_{k_1,t_1}\times R_{k_2,t_2})\neq \emptyset\}\,.\label{eq:def_r}
  \end{align}
\end{thm}
\begin{proof}[Proof of \Cref{thm:res_bound}]
  According to \Cref{thm:four},
  $$
  u(D,\mathcal M_R^2)=\inf h^{-1}\left[\sum_{i=1}^a \mu(A_i)-\sum_{j=1}^b\mu(B_j)\right]
  $$
  where the $\inf$ is taken over all $A_1,...,A_a,B_1,...,B_b\in \mathcal C_R^2$ and $h=1,2,3,...$ such that
  $$
  \sum_{i=1}^a I_{A_i}-\sum_{j=1}^b I_{B_j}\ge h I_D\,.
  $$
  
  Let $t=(t_1,t_2)$ be the pair of least common multiples of the moduli pairs 
  of the residue sets $A_1,...,A_a,B_1,...,B_b$.
  Then
  \begin{equation}\label{eq:constraint}
  h I_D \le \sum_{i=1}^a I_{A_i}-\sum_{j=1}^b I_{B_j}=\sum_{k_1=0}^{t_1-1}\sum_{k_2=0}^{t_2-1} d_{k_1,k_2} I_{R_{k_1,t_1}\times R_{k_2,t_2}}  
  \end{equation}
  for some integers $d_{0,0},d_{0,1},...,d_{t_1-1,t_2-1}$.
  
  
Thus  
  $$
  u(D,\mathcal M_R^2) =\inf h^{-1}\sum_{k_1=0}^{t_1-1}\sum_{k_2=0}^{t_2-1}d_{k_1,k_2} 
  \left(\frac{1}{t_1t_2}\right)
  =\inf h^{-1} \left(\frac{1}{t_1t_2}\right)
  \sum_{k_1=0}^{t_1-1}\sum_{k_2=0}^{t_2-1}d_{k_1,k_2}\,,
  $$
  where the inf is taken over all $(t_1,t_2)$ and $(d_{k_1,k_2}:k_1\in [t_1]-1,k_2\in [t_2]-1)$ such that \eqref{eq:constraint} holds.

  For given $(t_1,t_2)$, the right hand side of the above equation is minimized by setting $d_{k_1,k_2}=h$ if
  $D\cap (R_{k_1,t_1}\times R_{k_2,t_2})\neq\emptyset$ and $d_{k_1,k_2}=0$ otherwise.\end{proof}

\begin{lem}\label{lem:shift}
  If $(x,y)\in G$, then for every $n\in\mathbb N$
  there exists $a\in\mathbb N$ such that ${\rm gcd}(ax+y,n)=1$.
\end{lem}
\begin{proof}[Proof of \Cref{lem:shift}]
  Let $p_1,...,p_\ell$, $q_1,...,q_k$, $r_1,...,r_h$ be all distinct
  prime factors of $n$ such that
  \begin{align*}
    x&\equiv (0,...,0,a_1,...,a_k,c_1,...,c_h)~{\rm mod}~(p_1,...,p_\ell,q_1,...,q_k,r_1,...,r_h)\\
    y&\equiv (b_1,...,b_\ell,0,...,0,d_1,...,d_h)~{\rm mod}~(p_1,...,p_\ell,q_1,...,q_k,r_1,...,r_h)
  \end{align*}
  where $1\le a_j\le q_j-1$, $1\le c_j\le r_j-1$, $1\le d_j\le r_j-1$, $1\le b_j\le p_j-1$, for all $j$.
  
  Then one can pick any $a$ that satisfies
  $$
  a\equiv (0,...,0,1,...,1,0,...,0)~{\rm mod}~(p_1,...,p_\ell,q_1,...,q_k,r_1,...,r_h)\,.
  $$
  Existence of such an $a$ is guaranteed by the Chinese remainder theorem.
\end{proof}

\begin{lem}\label{lem:intersection}
  Let $k_1,k_2$ be two positive integers, and
   $(j_1,j_2)\in [k_1]\times [k_2]$.
   Then $G \cap (R_{j_1,k_1}\times R_{j_2,k_2})\neq \emptyset$ if and only if
   ${\rm gcd}(j_1,j_2,k_1,k_2)=1$.
\end{lem}
\begin{proof}[Proof of \Cref{lem:intersection}]
  The necessity is obvious. We only need to prove sufficiency.
  
  For $i=1,2$, let $p_i={\rm gcd}(k_i,j_i)$, $k_i=p_i r_i$, $j_i=p_i s_i$.
  By construction and the assumption that ${\rm gcd}(k_1,j_1,k_2,j_2)=1$
  we have
  \begin{align*}
    {\rm gcd}(p_1,p_2)={\rm gcd}(r_1,s_1)={\rm gcd}(r_2,s_2)=1\,.
  \end{align*}
  Then apply \Cref{lem:shift} to $(n,x,y)=(p_2,r_1,s_1)$, there exists
  $a_1\in\mathbb N$ such that
  \begin{equation}\label{eq:intersection1}
  {\rm gcd}(p_2,a_1 r_1+s_1)=1\,.
  \end{equation}
Apply \Cref{lem:shift} again to $(n,x,y)=(p_1(a_1r_1+s_1),r_2,s_2)$, there
exists an $a_2\in\mathbb N$ such that
  \begin{equation}\label{eq:intersection2}
  {\rm gcd}\left[p_1(a_1r_1+s_1),a_2 r_2+s_2\right]=1\,.
  \end{equation}
Now combine \eqref{eq:intersection1}, \eqref{eq:intersection2} and that ${\rm gcd}(p_1,p_2)=1$
we have
\begin{align*}
{\rm gcd}(a_1k_1+j_1,a_2k_2+j_2)&=1\,.\qedhere
\end{align*}
\end{proof}

\begin{thm}\label{thm:g_U_residual}
  $u(G,\mathcal M_R^2)=6/\pi^2\,.$
\end{thm}
\begin{proof}[Proof of \Cref{thm:g_U_residual}]
  Let $(k_1,k_2)\in\mathbb N^2$ and denote ${\rm cd}(k_1,k_2)$
  the set of prime common divisors of $k_1$ and $k_2$. Then \Cref{lem:intersection} implies that $G\cap (R_{j_1,k_1}\times R_{j_2,k_2})\neq \emptyset$ if and only if $j_1,j_2$ are not both divisible by any $p\in {\rm cd}(k_1,k_2)$. As a result,
  \begin{align*}
    \frac{r_{k_1,k_2}}{k_1k_2}=\prod_{p\in {\rm cd}(k_1,k_2)} (1-p_j^{-2})\,.
  \end{align*}
  
Now apply \Cref{thm:res_bound},  \begin{align*}
    u(G,\mathcal M_R^2)=&
    \inf_{k_1,k_2} \prod_{p\in{\rm cd}(k_1,k_2)}(1-p^{-2})=\prod_{p \text{ prime}}(1-p^{-2})=\frac{6}{\pi^2}\,.
    \qedhere
  \end{align*}
\end{proof} 

%% file: shift.tex
Combining \Cref{lem:hiararchy_n2} with \Cref{thm:rel_freq_G} and \Cref{thm:g_U_residual} we have
$$
u(G,\mathcal M_S^2)=6/\pi^2\,.
$$
Therefore, the proof of \Cref{thm:main_additive} will be complete if we can show
$$
\ell(G,\mathcal M_S^2)=0\,,
$$
which is the focus of the current section.

We prove the claim in a more general setting. Let $d\ge 2$ be a positive integer. For $1\le i\le d$, let $s_i:\mathbb N^d\mapsto\mathbb N^d$ be the shift operator in the $i$th coordinate:
$$
s_i(a_1,...,a_d)=(a_1,...,a_i+1,...,a_d)\,.
$$
We call a function $\mu:2^{\mathbb N^d}\mapsto \mathbb R$ shift-invariant if $\mu(A)=\mu(s_i^{-1}(A))$ for all $A\subseteq\mathbb N^{d}$ and all $1\le i\le d$.

Following ideas in \cite{schirokauer-kadane2007}, we study shift-invariant functions by constructing linear functionals on $\ell^\infty(\mathbb N^d)$ with certain desirable properties.  Recall the definition of $\ell^\infty(\mathbb N^d)$.
$$
\ell^\infty(\mathbb N^d)=\left\{x=\left(x(a)\in\mathbb R:a\in\mathbb N^d\right):\sup_{a\in\mathbb N^d} |x(a)|<\infty\right\}\,.
$$
Then $\ell^\infty(\mathbb N^d)$ is a Banach space equipped with norm $\|x\|=\sup_{a\in\mathbb N^d}|x(a)|$.
\begin{lem}\label{lem:linear}
  There exists a linear functional $\Phi$ on $\ell^\infty(\mathbb N^d)$ such that
  \begin{enumerate}
    \item $\Phi$ is shift-invariant: $\Phi(x)=\Phi(S_i x)$ for all $i\in [d]$, where $(S_i x)(a) = x(s_i(a))$ for all $a\in\mathbb N^d$;
    \item $\Phi$ is positive: $\Phi\ge 0$ whenever $x(a)\ge 0$ for all $a\in\mathbb N^d$;
    \item $\Phi$ is normalized: $\Phi(\mathbf 1)=1$ where $\mathbf 1$ is the constant-$1$ vector.
  \end{enumerate}
\end{lem}
\begin{proof}[Proof of \Cref{lem:linear}]
  Consider the linear subspace of $\ell^\infty(\mathbb N^d)$ given by
  $$
  W\stackrel{{\rm def}}{=}\left\{\sum_{i=1}^d (S_i x_i -x_i): x_i\in\ell^\infty(\mathbb N^d)\right\}\,.
  $$
  Then we can claim that $\overline W$ and $\mathbb R\mathbf 1$ intersect trivially. To see this, let $c\neq 0$ and $w=\sum_{i=1}^d (S_ix_i-x_i)$, then
  \begin{equation}\label{eq:norm}
    \|w+c\mathbf 1\|\ge \frac{1}{n^d}\sum_{a\in [n]^d}|w(a)+c|\ge \frac{1}{n^d}\left|\sum_{a\in [n]^d} (w(a)+c)\right|\rightarrow |c|
  \end{equation}
  as $n\rightarrow\infty$.
  
  Now let $\Phi_0$ be a linear functional on $W\oplus \mathbb R\mathbf 1$ given by
  $$
  \Phi_0(w+c\mathbf 1)=c\,.
  $$
  By \eqref{eq:norm}, $\|\Phi_0\|\le 1$.  By Hahn-Banach Theorem, there exists an extension $\Phi$ of $\Phi_0$ to $\ell^\infty(\mathbb N^d)$ such that $\|\Phi\|\le 1$.

Now we check that such a linear functional $\Phi$ satisfies the requirements of the claim.
\begin{enumerate}
  \item Shift-invariance: by linearity $\Phi(S_i x)-\Phi(x)=\Phi(S_ix-x)=0$.
  \item Normalized: by construction.
  \item Positivity: if $x\in\ell^\infty(\mathbb N^d)$ is positive, then we can write $x=cy$ for some $c>0$ and $\|y\|\le 1$, and 
  \begin{align*}
  \Phi(x)=&c\Phi(y) = c(1-\Phi({\bf 1}-y))\ge c(1-\|{\bf 1}-y\|) \ge 0\,. \qedhere    
  \end{align*}
\end{enumerate}
\end{proof}

The usefulness of \Cref{lem:linear} is the following general construction of shift-invariant probability measures on $2^{\mathbb N^d}$.

For $X,A\subseteq \mathbb N^d$, define $s^{-A}(X)=\bigcup_{a\in A}s^{-a}(X)$, where $s^{-a}(X)=s_1^{-a_1}(\cdots s_d^{-a_d}(X))$ for $a=(a_1,...a_d)\in\mathbb N^d$.

\begin{lem}\label{lem:construction}
  Let $\mu_1$ be a finitely additive probability on $2^{\mathbb N^d}$. Define
  $\mu:2^{\mathbb N^d}\mapsto \mathbb R$ as
  \begin{align*}
    \mu(Z) = \Phi\left\{\left[\mu_1(s^{-a}(Z)):a\in \mathbb N^d\right]\right\}\,.
  \end{align*}
  Then $\mu$ is a finitely additive, shift-invariant probability on $2^{\mathbb N^d}$.
\end{lem}
\begin{proof}[Proof of \Cref{lem:construction}]
  First $\mu_1(s^{-a}(Z))\in [0,1]$ for all $a\in\mathbb R^d$, by positivity and normalization of $\Phi$ we have $\mu(Z)\in[0,1]$ for all $Z$.
  
  Second, when $Z=\mathbb N^d$ we have $s^{-a}(Z)=\mathbb N^d$ for all $a\in\mathbb N^d$, and hence $\mu(Z)=\Phi(\mathbf 1)=1$.
  
  Third, if $Z_1,Z_2\subseteq\mathbb N^d$ are disjoint, then $s^{-a}(Z_1\cup Z_2)=s^{-a}(Z_1)\cup s^{-a}(Z_2)$ and $s^{-a}(Z_1)\cap s^{-a}(Z_2)=\emptyset$.
Then finite additivity of $\mu$ follows from linearity of $\Phi$.

Finally, for $i\in[d]$, $s^{-a}(s_i^{-1}(Z))=s^{-s_i(a)}(Z)$, so the shift-invariance of $\mu$ follows from the shift-invariance of $\Phi$ (Property 1 of \Cref{lem:linear}).
\end{proof}
\begin{lem}\label{lem:equivalent}
  For $X\subseteq \mathbb N^d$, the following are equivalent.
  \begin{enumerate}
    \item $s^{-A}(X)\neq \mathbb N^d$ for any finite set $A\subset \mathbb N^d$.
    \item There is a shift-invariant finitely-additive probability $\mu$ on $2^{\mathbb N^d}$ such that $\mu(X)=0$.
  \end{enumerate}
\end{lem}
\begin{proof}[Proof of \Cref{lem:equivalent}]
  ``$2\Rightarrow 1$'': If $\mu(X)=0$, then $\mu(s^{-a}(X))=0$ for every $a$.  Hence $\mu(s^{-A}(X))=0$ for any finite $A$.
  
  ``$1\Rightarrow 2$'': Let $\mathcal C$ be a family consisting of $\mathbb N^d$ and all sets of the form $s^{-A}(X)$ with finite $A$. Let $\mu_0:\mathcal C\mapsto \mathbb R^+$ be defined as $\mu_0(\mathbb N^d)=1$, $\mu_0(Y)=0$ if $Y\neq\mathbb N^d$.  The assumption that $s^{-A}(X)\neq \mathbb N^d$ for any finite set $A$ implies that, according to Theorem 1 of \cite{kadane-ohagan1995}, $\mu_0$ can be extended to $2^{\mathbb N^d}$. Let $\mu_1$ be such an extended finitely additive probability and let
  $$\mu(Z)=\Phi((\mu_1(s^{-a}(Z)):a\in\mathbb N^d))
  $$
  where $\Phi$ is the functional constructed in \Cref{lem:linear}.
  
  \Cref{lem:construction} ensures that $\mu$ is a shift-invariant finitely additive probability. On the other hand, $\mu_1(s^{-a}(X))=\mu_0(s^{-a}(X))=0$ for all $a\in\mathbb N^d$.  By construction, $\mu(X)=\Phi(\mathbf 0)=0$.
\end{proof}

\begin{thm}\label{thm:shift}
  $\ell(G,\mathcal M_S^2)=0$.
\end{thm}
\begin{proof}[Proof of \Cref{thm:shift}]
  According to \Cref{lem:equivalent}, it suffices to prove that $s^{-A}(G)\neq \mathbb N^2$ for every finite $A\subset \mathbb N^2$.
  
  Let $(a_1,b_1),(a_2,b_2),...,(a_m,b_m)$ be enumeration of all elements of $A$.  Let $p_1,...,p_m$ be $m$ arbitrary distinct prime numbers. By Chinese remainder theorem there exist $a,b\in \mathbb N$ such that
  \begin{align*}
  a +a_i\equiv &0~~{\rm mod}~~p_i,~~\forall~i\in[m]\,,\\
  b +b_i\equiv &0~~{\rm mod}~~p_i,~~\forall~i\in[m]\,.
  \end{align*}
Then $(a,b)\notin s^{-A}(G)$.
\end{proof}

%% file: countably.tex
In order to keep countable additivity in the probability, we must work with a smaller $\sigma$-field of subsets of $\mathbb N^2$.

For $i\in \mathbb N$ let $p_i$ be the $i$th prime number and define
 $$A_{i}=\{x\in\mathbb N: x\equiv 0 \mod p_i\}\,.$$
For finite disjoint subsets $I$, $J$ of $\mathbb N$ let
$$
A_{I,J}=\left(\cap_{i\in I} A_i\right)\bigcap\left(\cap_{i\in J} A_i^c\right)
$$
be the set of positive integers divisible by prime numbers in $I$ but not by those in $J$. It is allowed to have $I=J=\emptyset$, and we define $A_{\emptyset,\emptyset}=\mathbb N$.
Define
\begin{align*}
  \mathcal C=\left\{\bigcup_{k=1}^K A_{I_k,J_k}: K\in\mathbb N,I_k\cap J_k=\emptyset,|I_k|,|J_k|<\infty
  \right\}\bigcup \{\emptyset\}
\end{align*}


\begin{lem}\label{lemma:one}
$\mathcal{C}$ is a field of subsets of $\mathbb N$.
\end{lem}
\begin{proof}[Proof of \Cref{lemma:one}]
  Consider $\mathcal A=\{0,1\}^{\mathbb N}$.  For finite disjoint $I,J\subset\mathbb N$, we can represent
$A_{I,J}$ as a subset of $\mathcal A$ by $A_{I,J}\Leftrightarrow\{0\}^I \times \{1\}^J \times \{0,1\}^{(I\cup J)^c}$.
For example when $I=\{2\}$, $J=\{1,3\}$, then the corresponding subset of $\mathcal A$ is 
$\{x\in \{0,1\}^{\mathbb N}: x_1=0,x_2=1,x_3=0\}$, the cylinder in $\{0,1\}^{\mathbb N}$ with base $(0,1,0)$.

It is easy to check $\emptyset$ and $A_{\emptyset,\emptyset}=\mathbb N$ are in $\mathcal C$. We proceed to make the following three observations.
\begin{itemize}
\item[(a)] $\mathcal{C}$ is closed under finite unions.

 Let 
  $$
 C_1 = \bigcup^{K_1}_{k=1} A_{I^1_k,J^1_k} \mbox{ and } C_2 = \bigcup^{K_2}_{k=K_1+1} A_{I^2_k,J^2_k}.
 $$
and 
  $$
 \begin{array}{lcl}
 I_k = I^1_k \mbox{ and } J_k = J^1_k & \mbox{ for } & 1 \leq k \leq K_1\\
 I_k = I^2_k \mbox{ and } J_k = J^2_k & \mbox{ for } &K_1 + 1 \leq k \leq K_2.
 \end{array}
 $$
 Then
 $$
 C_1 \cup C_2 = \cup^{K_1+K_2}_{k=1} A_{I_k, J_k} \in \mathcal{C}.
 $$

 \item[(b)] $A^c_{I,J} \in \mathcal{C}$.
 
  Now assume $(I,J)\neq (\emptyset,\emptyset)$.
Use the product representation to write $A_{I,J}^c$:
$$
A_{I,J}^c = \left[\{0,1\}^{I\cup J}\backslash \left(\{0\}^I\times \{1\}^J\right)\right]\times \{0,1\}^{(I\cup J)^c}\,.
$$
and
$$\{0,1\}^{I\cup J}\backslash \{0\}^I\times \{1\}^J=\bigcup_{I'\subseteq (I\cup J),I'\neq I} \{0\}^{I'}\times\{1\}^{(I\cup J)\backslash I'}\,.$$
This shows that $A_{I,J}^c=\bigcup_{I'\subseteq (I\cup J),I'\neq I} A_{I',(I\cup J)\backslash I'}\in\mathcal C$.
\item[(c)] $A_{I_1,J_1} \cap A_{I_2,J_2} \in \mathcal{C}$.

For finite disjoint $(I_j,J_j)$, $(j=1,2)$, let $T=\cup(I_1,J_1,I_2,J_2)$. We consider the augmented representation of $A_{I_1,J_1}$ and $A_{I_2,J_2}$
\begin{align*}
  A_{I_1,J_1} = & \{0\}^{I_1}\times \{1\}^{J_1}\times \{0,1\}^{T\backslash(I_1\cup J_1)}\times \{0,1\}^{T^c}\\
  A_{I_2,J_2} = & \{0\}^{I_2}\times \{1\}^{J_2}\times \{0,1\}^{T\backslash(I_2\cup J_2)}\times \{0,1\}^{T^c}
\end{align*}
Let $B_j=\{0\}^{I_j}\times \{1\}^{J_j}\times \{0,1\}^{T\backslash(I_j\cup J_j)}$ for $j=1,2$.
Then each $B_j$ is a subset of $\{0,1\}^T$, which is a finite set.
Now let $C=B_1\cap B_2$, then $C$ is a subset of $\{0,1\}^T$. So there exists a subset $\mathcal I\subseteq T$, such that
$$
C=\bigcup_{I'\in\mathcal I} \{0\}^{I'}\times \{1\}^{T\backslash I'}.
$$
Since $T$ is finite, the union in the above expression for $C$ is finite. Thus we proved that
$A_{I_1,J_1}\cap A_{I_2,J_2} \in \mathcal C$.
%
%
%
%
%
\end{itemize}
The three observations (a-c) are sufficient to imply further claims such as that $\mathcal C$ is closed under complement, which concludes the proof.
 \end{proof}

\textbf{Remark.}
  Note that although $\mathcal C$ has an isomorphism between the subsets of $\mathbb N$ and those in $\mathcal A$,
  the generated $\sigma$-fields are different.
  In fact $\cap_{i=1}^\infty A_{\{i\},\emptyset}=\emptyset$, but $\cap_{i=1}^\infty \{0\}^{\{i\}}\times \{0,1\}^{\mathbb N \backslash\{i\}}=\{0\}^{\mathbb N}\neq \emptyset$.

Now we are ready to define the uniform probability measure on $\mathcal C$.
Let $P:\mathcal C\mapsto [0,1]$ be that
if $C=\cup_{k=1}^K A_{I_k,J_k}$ for disjoint sets $\{A_{I_k,J_k}:1\le k\le K\}$, then
$$
P\{C\}=\sum_{k=1}^K P\{A_{I_k,J_k}\}
$$
with
\begin{equation}\label{eq:unif1}
P\{A_{I,J}\}=\prod_{i\in I} p_i^{-1}\prod_{i\in J}(1-p_i^{-1})\,.
\end{equation}
We further define $P\{\emptyset\}=0$ and $P\{\mathbb N\}=1$.

Equation \eqref{eq:unif1} reflects the uniformity of $P$:  For distinct prime numbers $p$ and $q$
\begin{enumerate}
  \item [(i)] the probability of being divisible by a prime number $p$ is $p^{-1}$;
  \item [(ii)]  being divisible by $p$ and being divisible by $q$ are independent events.
\end{enumerate}
 
\begin{thm}\label{thm:one}
  $P$ is a probability on $\mathcal C$ and can be uniquely extended to $\mathcal F=\sigma(\mathcal C)$.
\end{thm}
\begin{proof}[Proof of \Cref{thm:one}]
  We only need to prove countable additivity of $P$ on $\mathcal C$.  The second part follows from Carath\'{e}odory's extension.
  
  Let $A_{I,J}=\bigcup_{k=1}^\infty A_{I_k,J_k}$, where $\{A_{I_k,J_k}:k\ge 1\}$ are disjoint with $I_k$, $J_k$ finite and disjoint.
  Now define $Q$ to be the product measure on $\mathcal A$ with marginal $Q_i$ being Bernoulli$(1-p_i^{-1})$.
The existence and  uniqueness of $Q$ is guaranteed by Kolmogorov's extension.
  
  Then $P$ and $Q$ agree on $\mathcal C$.  Since $Q$ is a probability measure we have
  \begin{align*}
    P\{A_{I,J}\}=&Q\{A_{I,J}\}=\sum_{k=1}^\infty Q\{A_{I_k,J_k}\}=\sum_{k=1}^\infty P\{A_{I_k,J_k}\}\,.\qedhere
  \end{align*} 
\end{proof}

%

\begin{thm}\label{thm:two}
Let $P_2$ be the product measure of $P$ on $\mathbb N^2$. Then  $$
  P_2\{G\}=6/\pi^2\,.
  $$
\end{thm}
\begin{proof}[Proof of \Cref{thm:two}]
  ${\rm gcd}(x,y)=1$ if and only if $(x,y)\in (A_{\{i\},\emptyset}\times A_{\{i\},\emptyset})^c$ for all $i$. By independence between $A_{\{i\},\emptyset}$ as $i$ changes,
  \begin{align*}
    P_2\{G\}=&\prod_{i=1}^\infty (1-p_{i}^{-2})=\frac{6}{\pi^2}\,. \qedhere
  \end{align*}
  %
\end{proof}

%% file: conclusion.tex
The probability assigned to the set of $G$ of relatively prime integers depends on the sense of ``uniform'' probability being used.  When the class of finitely additive probabilities defined by relative frequency is used, $P(G)=6/\pi^2$.  Similarly, when the countably additive probability defined in \Cref{sec:countably}, $P\{G\}=6/\pi^2$ is the only value supported.  However, when the finitely additive classes defined by shift invariance or residue classes are involved, there are elements of those classes satisfying $P\{G\}=x$ if and only if $x\in[0,6/\pi^2]$.